\documentclass[submission]{FPSAC2022}
\usepackage{pack}

\title{A description of the minimal elements of Shi regions in classical Weyl Groups}

\author{Balthazar Charles\addressmark{1}\thanks{\href{mailto:balthazar.charles@universite-paris-saclay.fr}{balthazar.charles@universite-paris-saclay.fr}}}

\address{\addressmark{1}Université Paris-Saclay, CNRS, Laboratoire Interdisciplinaire des Sciences du Numérique, 91405, Orsay, France}

\abstract{In this extended abstract, we show how a bijection between parking functions and regions of the Shi arrangement from [Athanasiadis, Linusson '99] (in type $A_n$) and [Armstrong, Reiner, Rhoades '15] (in type $B_n, C_n, D_n$) allows for the computation of the minimal elements of the Shi regions. This gives a combinatorial interpretation of these minimal elements: they can be seen as counting non-crossing arcs in non-nesting arc diagrams.}


\received{\today}

\usepackage[backend = bibtex]{biblatex}
\begin{document}

\maketitle

The Shi arrangements are a well studied subject in algebraic combinatorics, and their investigation has generated a number of results of bijective nature (for instance \cite{fishel2010bijection}, \cite{levear2020bijection}, \cite{tzanaki2015bijections}).
Recently, a push has been made to understand the minimal elements of Shi regions in connection with \cite[Conjecture 2]{dyer2016small} from Dyer \& Hohlweg. In the case of a crystallographic root system $\Phi$, 
Shi gives in \cite{JYS1} a description of the elements of the affine Weyl group associated to $\Phi$ as vectors in $\ZZ^{\Phi^+}$.
Shi \cite{JYS2} also proves in the crystallographic case that the Shi regions can be described by a so called \emph{sign type} and that they have an unique minimal element. Given these two results, we ask ourselves how to compute the minimal element of a given sign type.

In \cite{athanasiadis1999simple}, Athanasiadis \& Linusson give a simple bijection between \emph{parking functions} and Shi regions of type $A_n$, specifically as described by their sign type. In \cite{reiner1997non}, Armstrong, Reiner \& Rhoades extend this bijection to all crystallographic root systems.
We show in this extended abstract that this bijection can, in the classical type $A_n, B_n, C_n, D_n$ be used to precisely describe the minimal element of each Shi region. This description is essentially combinatorial: in the simplest case $A_n$, if the parking function is described as a permutation together with a non-nesting partition, the coefficient of $e_i-e_j$ is the number of non-crossing arcs between the values $i$ and $j$.

Although we assume some familiarity with crystallographic root systems and Weyl groups, we give some light background in Section \ref{sec:shi} to fix the notations. In Section \ref{sec:A} we describe the minimal elements in type $A_n$ with a detailed proof. Finally, in Section \ref{sec:gen}, we discuss how to extend the result in type $B_n, C_n, D_n$ with sketches of proofs.


\section{Affine Weyl groups and Shi regions.}\label{sec:shi}

\subsection{Affine Weyl Groups}\label{sub:affine}
Let $V$ be an Euclidean space with inner product $\langle \cdot \sep \cdot \rangle$. Let $\Phi$ be an irreducible crystallographic root system in $V$. Additionally, we suppose that $\Phi$ spans $V$.
For $\alpha \in \Phi$, $k \in \ZZ$, consider the orthogonal affine reflections:
\[s_{\alpha, k} = x \longmapsto x - 2\left(\langle x \sep \alpha \rangle -k\right)\frac{\alpha}{\langle \alpha, \alpha \rangle}.\]
The \emph{Weyl group} $W$ (resp. \emph{affine Weyl group} $\widetilde{W}$) associated with $\Phi$ is the group generated by $\{s_{\alpha, 0} \sep \alpha \in \Phi\}$ (resp. $\{s_{\alpha, k} \sep \alpha \in \Phi, k \in \ZZ\}$).

Let $f$ be a linear form such that $\Phi\cap\ker f = \emptyset$. This gives a choice of \emph{positive roots} $\Phi^+ = \Phi \cap f\inv\{\RR^+\}$ and of \emph{simple roots}, defined as the roots in $\Phi^+$ that generate the extreme rays of $\cone(\Phi^+) = \sum_{\alpha \in \Phi^+} \RR^+\alpha$. We denote the set of simple roots by $\Delta$.
Because we supposed that $\Phi$ spans $V$, $\Delta$ is a basis of $V$.

Since $\Phi$ is crystallographic, it has the property that $\Phi^+ \subset \NN\Delta$. This allows for the definition of the \emph{root poset} on $(\Phi^+, \leq)$ where for $\alpha, \beta \in \Phi^+$ we say $\alpha \leq \beta$ if $\beta - \alpha \in \NN\Delta$. This poset is graded by the \emph{height} function $h: \Phi^+ \longrightarrow \NN$ where $h(\alpha)$ is the sum of the coefficients in the (unique) decomposition of $\alpha$ over $\Delta$.
The root poset has a unique maximal element called the \emph{highest root} that will be denoted by $\alpha_0$.

\subsection{Alcoves and Shi relations}\label{sub:shi_rel}
Consider the collections of affine hyperplanes (respectively, half-spaces):
\[H_{\alpha, k} = \{x \in V \sep \langle \alpha \sep x \rangle = k\}, 
\quad H^+_{\alpha, k} = \{x \in V \sep \langle \alpha \sep x \rangle > k\},
\quad H^-_{\alpha, k} = \{x \in V \sep \langle \alpha \sep x \rangle < k\}\]
for $\alpha \in \Phi^+$ and $k \in \ZZ$. We denote the arrangement of the $H_{\alpha, k}$ hyperplanes by:
\[\mathcal{A} = \bigcup_{\alpha \in \Phi^+, k \in \ZZ} H_{\alpha, k}.\]
The connected components of $V \setminus \mathcal{A}$ are called \emph{alcoves}. The \emph{fundamental chamber} and the \emph{fundamental alcove} are, respectively:
\[C = \bigcap_{\alpha \in \Delta} H^+_{\alpha, 0},\qquad A_e = C \cap H^-_{\alpha_0, 1}.\]
It is well known that the set of alcoves is in bijection with the affine Weyl Group $\widetilde{W}$, the bijection being $w \mapsto wA_e$. Given $w \in \widetilde{W}$ and its corresponding alcove $A_w$, define:
\[K(w) = (k(w, \alpha))_{\alpha \in \Phi^+} \textrm{ where } k(w, \alpha) = \max\{i \in \ZZ \sep A_w \subset H^+_{\alpha, i}\}.\]
The map $K$ is an injection of $\widetilde{W}$ in $\ZZ^{\Phi^+}$ as it essentially is a description of the so-called \emph{inversion set} of $w$ (see \cite{JYS1}). We are interested in the elements of $\ZZ^{\Phi^+}$ of the form $K(w)$ which we call \emph{Shi vectors}. Shi vectors are characterized by Shi: 

\begin{thm}\cite[Theorem 1.1]{JYS3}\label{thm:shi_relations}
	Consider $v \in \ZZ^{\Phi^+}$. There exists some $w \in \widetilde{W}$ such that $K(w) = v$ if and only if, for all $\alpha, \beta, \gamma \in \Phi^+$ such that $\alpha + \beta = \gamma$, we have $v_{\alpha} + v_{\beta} + \varepsilon_{\alpha, \beta} = v_{\gamma}$ for some $\varepsilon_{\alpha, \beta} \in \{0, 1\}$. We call these equations the \emph{Shi relations}.
\end{thm}

\begin{rmk}
	We want to bring the attention of the reader to the fact that in Shi's original formulation given as reference, the condition on $\alpha, \beta, \gamma$ in the previous theorem is $\alpha^{\vee} + \beta^{\vee} = \gamma^{\vee}$ where for some non zero vector $v \in V$, $v^{\vee}$ is defined as $2v/\langle v \sep v \rangle$. However, this is because the $H_{\alpha, k}$ are instead defined as $\{x \in V \sep \langle \alpha^{\vee} \sep x \rangle = 0\}$. In our convention, the successive applications of $\alpha \mapsto {\alpha}^{\vee}$ "cancel out", giving this formulation.
\end{rmk}

\subsection{Shi arrangement and Shi regions}\label{sub:regions}
\begin{dftn}
	We denote by $\mathcal{A}_1$ the \emph{Shi arrangement} defined by:
	\[\mathcal{A}_1 = \bigcup_{\alpha \in \Phi^+} (H_{\alpha, 0} \cup H_{\alpha, 1}).\]
	The connected components of $V \setminus \mathcal{A}_1$ are called the \emph{Shi regions}.
\end{dftn}

The Shi arrangement $\mathcal{A}_1$ being a sub-arrangement of $\mathcal{A}$ which defines the alcoves, for a Shi region $R$ we have $\overline{R} = \bigcup_{A_w \cap R \neq \emptyset} \overline{A_w}$: said more loosely, a Shi region is a union of alcoves. Using this fact, for $w \in \widetilde{W}$ we (abusively) write that $w \in R$ to mean $A_w \subset R$.

Setting, for a real number $x$, $\sign(x) = -$ if $x < 0$, $\sign(x) = 0$ if $x = 0$ and $\sign(x) = +$ if $x > 0$, we can define $\sign(w)$ for $w \in \widetilde{W}$ as $(\sign(k(w, \alpha)))_{\alpha \in \Phi^+}$. We have the following (see \cite{JYS2} for a discussion):
\begin{prop}
	Let $R$ be a Shi region. The value of $\sign(w)$ is constant for all $w \in R$. The \emph{sign type} of $R$, denoted by $\sign(R)$, is defined as $\sign(w)$ for any $w \in R$. The map $R \mapsto \sign(R)$ is an injection of the set of Shi regions in $\{-,0,+\}^{\Phi^+}$.
\end{prop}

\begin{ex}
	A consequence of Theorem \ref{thm:shi_relations} is that not all sign types are possible. In type $A_2$ for instance, we have 3 roots $e_1-e_2, e_2-e_3, e_1-e_3$.
	Giving the signs in the order $(e_1-e_2, e_1-e_3, e_2-e_3)$ (as we do everywhere a $A_2$ sign type appears in the remainder of this paper), the $A_2$ sign types are the following:
	\[\begin{Bmatrix}
	(+,+,+), &(-,-,-), &(+,+,-), &(-,+,+), &(-,-,+), &(+,-,-),\\
	(+,+,0), &(0,+,+), &(+,0,-), &(-,0,+), &(-,-,0), &(0,-,-),\\ 
	& (0,+,0), &(0,0,-), &(-,0,0), &(0,0,0) &
	\end{Bmatrix}.\]
	Notice that replacing all the zeros of a possible $A_2$ sign type with pluses gives one of the possible $A_2$ sign type of the first row above. This is an important example as we will, in Sections \ref{sec:A} and \ref{sec:gen}, interpret triples of signs as an $A_2$ sign type.
\end{ex}

Given the sign type $\sign(R)$ of some Shi region $R$, it seems natural to ask to exhibit an element $w \in \widetilde{W}$ such that $\sign(w) = \sign(R)$. We can actually do slightly better than exhibit any element: in \cite[Proposition 7.2]{JYS2}, Shi proves that every Shi region has a \emph{minimal element}. We set out to describe these minimal elements in Sections \ref{sec:A} and \ref{sec:gen}.

\begin{prop}\label{prop:min_shi}
	Let $R$ be a Shi region. There exists a unique minimal element $\min(R) \in R$ in the sense that for every $w \in R$ and every $\alpha \in \Phi^+$, $|k(\min(R), \alpha)| \leq |k(w, \alpha)|$.
\end{prop}

\begin{rmk}
	The minimal element can be defined in other equivalent ways: it is also the unique element such that $\sum_{\alpha \in \Phi^+} |k(w, \alpha)|$ is minimized, or the unique minimal element of $R$ for the \emph{weak order} of $\widetilde{W}$. We refer to \cite[Section 7]{JYS2} for details. For our application, the version of Proposition \ref{prop:min_shi} is the most convenient.
\end{rmk}


\section{Type $A_{n}$}\label{sec:A}

In this Section we systematically refer to root systems, Weyl groups, Shi regions... of type $A_{n}$, with the following usual conventions for the underlying root system $\Phi$:
\[\Phi^+ = \{e_i -e_j \sep 1 \leq i < j \leq n+1\}, \qquad \Delta = \{e_i-e_{i+1} \sep 1 \leq i \leq n\}.\]
If $v$ is a vector indexed by $\Phi^+$, we will shorten $v_{e_i-e_j}$ by $v_{i,j}$. It will sometimes be useful to write $v_{i,j}$ as $v_{j,i}$: this is non ambiguous as $e_i-e_j \in \Phi^+$ if and only if $i<j$.

As usual, we have a realization of $A_{n}$ as $\symm_{n+1}$. Let $\pi$ be a permutation of $\intint{1, n+1}$. If $a, b \in \intint{1, n+1}$ are distinct, $c = \pi\inv(a)$ and $d=\pi\inv(b)$, we define an \emph{arc} on $\pi$ as $\{(a, c), (b, d)\}$. Because when $\pi$ is fixed, knowing $\{a, b\}$ and knowing $\{c, d\}$ is the same, we will designate $\{(a, c), (b, d)\}$ either by $(a, b)$ or $(b, a)$ if we are interested in the values or, if we are interested in the positions, $[c, d]$ or $[d, c]$.
An \emph{arc diagram} is a pair $(\pi, P)$ where $\pi \in \symm_{n+1}$ and $P$ is a partition of $\intint{1, n+1}$. If $P = \{P_1, ..., P_k\}$ and $P_i = \{p_{i, 1}, ..., p_{i, k_i}\}$ with $p_{i, j} \leq p_{i, j+1}$ for all $i \in \intint{1, k}$, $j \in \intint{1, k_i-1}$, the \emph{set of arcs} of $(\pi, P)$ is:
\[\arcs(\pi, P) = \{[p_{i, j}, p_{i, j+1}] \sep i \in \intint{1, k}, j \in \intint{1, k_i-1}\}.\]
Note that we can recover $P$ from the set of arcs, meaning we can graphically represent an arc diagram as in Figure \ref{fig:bijA}.
Two arcs $[a, b]$ and $[c, d]$ with $a<b, c<d$ and $a<c$ are \emph{crossing} if $a < c < b < d$. They are \emph{nesting} (in what we call a "\emph{pictorial}" sense) if $a < c < d < b$. An arc diagram is \emph{non-crossing} (resp. \emph{non-nesting}) if no two of its arcs are crossing (resp. nesting).

\subsection{The Athanasiadis-Linusson bijection with parking functions}
In this Section, we present a bijection between Shi regions of type $A_{n-1}$ and certain types of arc diagrams given by Athanasiadis \& Linusson in \cite{athanasiadis1999simple}.

Consider a Shi region of type $A_{n-1}$ given by its sign type $v = (v_{i, j})_{1\leq i<j \leq n}$. Construct the vector $I = (|\{v_{i, j}= 0 \textrm{ or } +\sep j \in \intint{1,n}\}|)_{1 \leq i < n}$. From \cite{athanasiadis1999simple}, this is the inversion vector of some  $\pi \in \symm_n$, meaning that there exists a unique $\pi \in \symm_n$ such that $I_i$ is equal to the number of values lower than $i$ that appear to the left of $i$\footnote{This is essentially a Lehmer code}. Finally, for every $1\leq i<j \leq n$, if $v_{i, j} = +$, add an arc between $i$ and $j$, removing any arc that contains another.

\vspace{-1mm}
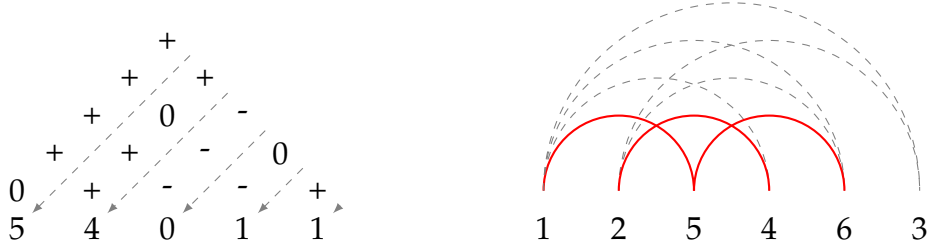
\begin{figure}[h]
	\centering
	\begin{tikzpicture}
%
%
%
%

\node at (0,-1/2){5};
\node at (1,-1/2){4};
\node at (2,-1/2){0};
\node at (3,-1/2){1};
\node at (4,-1/2){1};

\node at (0,0/2){0};
\node at (1,0/2){+};
\node at (2,0/2){-};
\node at (3,0/2){-};
\node at (4,0/2){+};

\node at (0.5,1/2){+};
\node at (1.5,1/2){+};
\node at (2.5,1/2){-};
\node at (3.5,1/2){0};

\node at (1,2/2){+};
\node at (2,2/2){0};
\node at (3,2/2){-};

\node at (1.5,3/2){+};
\node at (2.5,3/2){+};

\node at (2,4/2){+};

\node at (7,-1/2){1};
\node at (8,-1/2){2};
\node at (9,-1/2){5};
\node at (10,-1/2){4};
\node at (11,-1/2){6};
\node at (12,-1/2){3};

\draw[->,>=latex, dashed, gray] (2.3,1.8) to (0.2,-0.3);
\draw[->,>=latex, dashed, gray] (2.8,1.3) to (1.2,-0.3);
\draw[->,>=latex, dashed, gray] (3.3,0.8) to (2.2,-0.3);
\draw[->,>=latex, dashed, gray] (3.8,0.3) to (3.2,-0.3);
\draw[->,>=latex, dashed, gray] (4.3,-0.2) to (4.2,-0.3);

\draw(0+7,0) [gray, dashed] arc[radius=2.5, start angle=180, end angle=0];
\draw(0+7,0) [gray, dashed] arc[radius=2, start angle=180, end angle=0];
\draw(0+7,0) [red, thick] arc[radius=1, start angle=180, end angle=0];
\draw(0+7,0) [gray, dashed] arc[radius=1.5, start angle=180, end angle=0];

\draw(1+7,0) [gray, dashed] arc[radius=2, start angle=180, end angle=0];
\draw(1+7,0) [gray, dashed] arc[radius=1.5, start angle=180, end angle=0];
\draw(1+7,0) [red, thick] arc[radius=1, start angle=180, end angle=0];

\draw(2+7,0) [red, thick] arc[radius=1, start angle=180, end angle=0];
\end{tikzpicture}
	\caption{Example of computation of the Athanasiadis-Linusson bijection. The sign type $v$ is given as a pyramid: $v_{i,j}$ is the $i$-th sign from the left on the $j-i$-th row from the bottom. For instance, the "middle 0" is $v_{2,5}$.}
	\label{fig:bijA}
\end{figure}
\vspace{-5mm}

By construction, the set of arcs is non-nesting, and any block of the partition is sorted (meaning if the positions $p < q$ are in the same block, then $\pi(p) < \pi(q)$).

\begin{thm}\cite[Theorem 2.2]{athanasiadis1999simple}\label{thm:AL}
	This defines a bijection between Shi regions in type $A_{n-1}$ and non-nesting arc diagrams with sorted blocks on $\intint{1, n}$. For the purpose of this paper, we call such diagrams \emph{(type $A_{n-1}$) parking functions}.
\end{thm}

\subsection{The minimal elements}
In this section, we prove the following result:
\begin{prop}\label{prop:shi_min_A}
	Let $R$ be a Shi region, $v = \sign(R)$.
	Let $(\pi, P)$ be the parking function associated to $v$. We define $\eta_{i,j}$ to be the maximal number of non-crossing arcs of $P$ between the values $i$ and $j$. The vector $m \in \ZZ^{\Phi^+}$ defined as follows is $R$'s minimal element:
	\[m_{i, j} = \left\{\begin{aligned}
	\eta_{i,j} &\textrm{ if } v_{i,j} = 0, +\\
	-(\eta_{i,j}+1) &\textrm{ if } v_{i,j} = - \\
	\end{aligned}\right..\]
\end{prop}

This, together with Theorem \ref{thm:AL}, give an explicit bijection between parking functions and the minimal elements of the Shi regions. The following lemma is the main ingredient for proving our Proposition.

\begin{lemme}\label{lem:arcs}
	Let $P$ be a non-nesting partition of $\intint{1, n}$ (meaning it forms a non-nesting arc diagram with the identity permutation). Let $\eta \in \NN^{\Phi^+}$ as defined in Proposition \ref{prop:shi_min_A}. Then for any $1 \leq a < b < c \leq n$, $\eta_{a,c} = \eta_{a,b} + \eta_{b,c} + \varepsilon_{\eta}$, with $\varepsilon_{\eta} \in \{0,1\}$.\footnote{Although $\varepsilon_{\eta}$ depends on $a, b, c$, we do not signify it by the notation, as it is always clear in context.}
\end{lemme}

\begin{proof}
	For $i, j \in \intint{1, n}$, let $S_{i,j}$ be a set of non-nesting, non-crossing arcs between $i$ and $j$ of maximal cardinality. It is clear that $\eta_{a, c} \geq \eta_{a,b} + \eta_{b,c}$ since $S_{a,b} \cup S_{b,c}$ is a set of non-crossing, non-nesting arcs. Conversely, in $S_{a,c}$, by definition of $\eta$, at most $\eta_{a,b}$ arcs are between $a$ and $b$ and similarly for $b$ and $c$. Any arc in $S_{a,b}$ that is not between $a$ and $b$ or $b$ and $c$ must therefore be of the form $(d,e)$ with $d<b<e$ and two such arcs must be either nesting or crossing.
	Thus $S_{a,b}$ contains at most $\eta_{a,b} + \eta_{b,c}+1$ element.
\end{proof}

\vspace{-4mm}
\begin{figure}[h!]
	\centering
	\begin{tikzpicture}
\draw(0,0) -- ++ (7,0);
\draw[dashed] (0,0)--(0,2);
\draw[dashed] (4,0)--(4,1.5);
\draw[dashed] (7,0)--(7,2);
\draw[latex-latex] (0,2)--(7,2);
\draw(3.5,2)node[above]{$\eta_{a,c}$};
\draw[latex-latex] (0,1.5)--(4,1.5);
\draw(2,1.5)node[below]{$\eta_{a,b}$};
\draw[latex-latex] (4, 1.5)--(7,1.5);
\draw(5.5,1.5)node[below]{$\eta_{b,c}$};
\foreach \x/\y in {0/a, 4/b, 7/c}{
	\draw[circle,fill] (\x,0)circle[radius=1mm]node[below]{$\y$};
}
\draw(0,0) [red, very thick] arc[radius=1, start angle=180, end angle=0];
\draw(2,0) [red, very thick] arc[radius=1, start angle=180, end angle=0];
\draw(4,0) [blue, very thick] arc[radius=1, start angle=180, end angle=0];
\draw(6,0) [blue, very thick] arc[radius=0.5, start angle=180, end angle=0];
\draw(3,0) [black] arc[radius=1, start angle=180, end angle=0];

\draw(9,0) -- ++ (7,0);
\foreach \x/\y in {0/a, 3/b, 7/c}{
	\draw[circle,fill] (\x+9,0)circle[radius=1mm]node[below]{$\y$};
}

\draw(9,0) -- ++ (7,0);
\draw[dashed] (9,0)-- ++(0,2);
\draw[dashed] (12,0)-- ++(0,1.5);
\draw[dashed] (16,0)-- ++(0,2);
\draw[latex-latex] (9,2)-- (16,2);
\draw(12.5,2)node[above]{$\eta_{a,c}$};
\draw[latex-latex] (9,1.5)--(12,1.5);
\draw(10.5,1.5)node[below]{$\eta_{a,b}$};
\draw[latex-latex] (12, 1.5)--(16,1.5);
\draw(14,1.5)node[below]{$\eta_{b,c}$};

\draw( 9,0) [red, very thick] arc[radius=1, start angle=180, end angle=0];
\draw(11,0) [magenta, very thick] arc[radius=1, start angle=180, end angle=0];
\draw(13,0) [blue, very thick] arc[radius=1, start angle=180, end angle=0];
\draw(15,0) [blue, very thick] arc[radius=0.5, start angle=180, end angle=0];
\draw(12,0) [black] arc[radius=1, start angle=180, end angle=0];
\end{tikzpicture}
	\vspace{-8mm}
	\caption{Illustration of Lemma \ref{lem:arcs}: the arcs of $S_{a,b}$ are represented in bold. On the left, the case where $\varepsilon_{\eta} = 0$, on the right, the case where $\varepsilon_{\eta} = 1$. Note that we need not to choose the same set of non-nesting non-crossing arcs to compute $\eta_{a,b}, \eta_{b,c}$ and $\eta_{a,c}$.}
	\label{fig:arcs_giv_shi}
\end{figure}
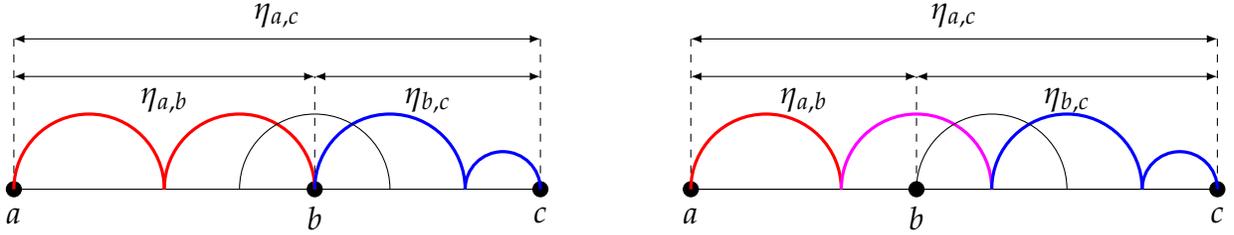

\vspace{-6mm}

\begin{proof}[Proof of Proposition \ref{prop:shi_min_A}]
	Firstly, we show that the vector $m$ is a Shi vector, meaning it verifies the Shi relations $m_{a,c} = m_{a, b} + m_{b,c} + \varepsilon_{m}$\footnote{Same remark as footnote 2.} for $1 \leq a < b < c \leq n+1$. Interpret $(v_{a,b}, v_{a,c}, v_{b,c})$ as an $A_2$ sign type: applying the Athanasiadis-Linusson bijection, we associate it to a permutation of $\symm_{\{a,b,c\}}$ called its \emph{pattern}. The pattern is unchanged when replacing all zeros with pluses in a sign type, leaving 6 cases to check: we do so in Table \ref{fig:cases} (patterns are given on $\{a,b,c\} = \{1,2,3\}$ for simplicity). 
	\begin{table}[h!]
		\begin{tabular}{ p{18mm} | c || c: p{0mm} c p{0mm} c p{0mm} c}			
			Sign type & Pattern & $m_{1,3}$ & = & ${\red m_{1,2}}$ & + & ${\blue m_{2,3}}$ & + & ${\magenta \varepsilon_m}$\\
			\hline
			$(+,+,+)$ & $1,2,3$ & $\eta_{1,2} + \eta_{2,3} + \varepsilon_{\eta}$ & = & ${\red \eta_{1,2}}$ & + & ${\blue \eta_{2,3}}$ & + & ${\magenta \varepsilon_{\eta}}$\\
			$(+,+,-)$ & $1,3,2$ & $\eta_{1,2} - \eta_{2,3} - \varepsilon_{\eta}$ & = & ${\red \eta_{1,2}}$ & + & ${\blue -(\eta_{2,3}+1)}$ & + & ${\magenta 1-\varepsilon_{\eta}}$\\
			$(-,+,+)$ & $2,1,3$ & $\eta_{2,3} - \eta_{1,2} - \varepsilon_{\eta}$ & = & ${\red -(\eta_{1,2}+1)}$ & + & ${\blue \eta_{2,3}}$ & + & ${\magenta 1-\varepsilon_{\eta}}$\\
			$(-,-,+)$ & $2,3,1$ & $-(\eta_{1,2} - \eta_{2,3} - \varepsilon_{\eta} + 1)$ & = & ${\red -(\eta_{1,2}+1)}$ & + & ${\blue \eta_{2,3}}$ & + & ${\magenta \varepsilon_{\eta}}$\\
			$(+,-,-)$ & $3,1,2$ & $-(\eta_{3,2} - \eta_{1,2} - \varepsilon_{\eta} + 1)$ & = & ${\red \eta_{1,2}}$ & + & ${\blue -(\eta_{2,3}+1)}$ & + & ${\magenta \varepsilon_{\eta}}$\\
			$(-,-,-)$ & $3,2,1$ & $-(\eta_{3,2} + \eta_{1,2} + \varepsilon_{\eta} + 1)$ & = & ${\red -(\eta_{1,2}+1)}$ & + & ${\blue -(\eta_{2,3}+1)}$ & + & ${\magenta 1-\varepsilon_{\eta}}$\\
		\end{tabular}
		\caption{In each case we translate the Shi relation according to our Proposition \ref{prop:shi_min_A}. Notice that everytime,  the $\varepsilon_{\eta}$ given by Lemma \ref{lem:arcs} gives $\varepsilon_m \in \{0,1\}$: $m$ is a Shi vector.}
		\label{fig:cases}
		\vspace{-3mm}
	\end{table}
	
	Secondly, we show that $m$ is minimal. 
	Let $1 \leq a < b \leq n+1$ and $c = \pi\inv(a), d = \pi\inv(b)$. We proceed by strong induction on $|c-d|$. If $|c-d| = 1$, the definition of the Athanasiadis-Linusson bijection authorizes only three cases:
	either $c < d, \eta_{a, b} = 1$ so $m_{a,b} = 1$; $c < d, \eta_{a, b} = 0$ so $m_{a,b} = 0$ or $c > d, \eta_{a, b} = 0$ so $m_{a,b} = -1$. In all 3 cases, this is the minimal possible value that respects the sign type.
	Suppose now that $|c-d| = k > 1$. If $S_{a,b} = \{(a, b)\}$ or $S_{a,b} = \emptyset$ we are in the same situation as before and $m$ could not be smaller. Otherwise, we may choose an extremity $e$ of some arc in $S_{a,b}$, strictly between the positions $c$ and $d$ such that Lemma \ref{lem:arcs} applied to ${a,e,b}$ yield $\varepsilon_{\eta} = 0$. Depending on the pattern formed by $a,e,b$, we have again six cases, which we check in Table \ref{fig:cases_min}.
	\begin{table}[h!]
		\centering
		\begin{tabular}{ c | c || c:cccccc }			
			$a,e,b$ & Sign type & $m_{a,b}$ & & $m_{a,e}$ & &  $m_{e,b}$ & &$\varepsilon_m$\\
			\hline
			$1,2,3$ & $(+,+,+)$ & $\eta_{1,2} + \eta_{2,3}$ & = +& $\eta_{1,2}$ &+& $\eta_{2,3}$ &+& $0$\\
			$1,3,2$ & $(+,+,-)$ & $\eta_{1,3} + \eta_{3,2}$ & = +& $\eta_{1,3}$ &-& $-(\eta_{2,3}+1)$ &-& $1$\\
			$2,1,3$ & $(-,+,+)$ & $\eta_{2,1} + \eta_{1,3}$ & = -& $- (\eta_{1,2}+1)$ &+& $\eta_{1,3}$ &-& $1$\\
			$2,3,1$ & $(-,-,+)$ & $-(\eta_{2,3} + \eta_{3,1} + 1)$ & = -& $\eta_{2,3}$ &+& $-(\eta_{1,3}+1)$ &-& $0$\\
			$3,1,2$ & $(+,-,-)$ & $-(\eta_{3,1} + \eta_{1,2} + 1)$ & = +& $-(\eta_{2,3}+1)$ &-& $\eta_{1,2}$ &-& $0$\\
			$3,2,1$ & $(-,-,-)$ & $-(\eta_{3,2} + \eta_{2,3} + 1)$ & = +& $-(\eta_{2,3}+1)$ &+& $-(\eta_{1,2}+1)$ &+& $1$\\
		\end{tabular}
		\caption{The formula for $m$ gives a minimal vector. We first express $m_{a,b}$ as given by our Proposition, and then using the Shi relations. Notice the last column: in every case, depending on the sign of $m_{a, b}$ and whether $\varepsilon_m$ is added or subtracted, $\varepsilon_m$ takes the value that minimizes the absolute value of $m_{a, b}$.}
		\label{fig:cases_min}
	\end{table}\\
	If we make the induction hypothesis that our formula gives the minimal possible coefficient for $m_{i,j}$ with $|\pi\inv(i)-\pi\inv(j)| < k$, then $m_{a,e}$ and $m_{e,b}$ are minimal and, from Table \ref{fig:cases_min}, so is $m_{a,b}$.
\end{proof}


\section{Generalization to classical types}\label{sec:gen}

\subsection{Generalized parking functions}
In this Section it will be useful to see a Shi region $R$ with sign type $v$ as defined by inequalities. Specifically, $R$ can be defined as the set of vectors $x \in V$ such that:
\[\forall \alpha \in \Phi^+, \quad \langle x|\alpha\rangle < 0 \textrm{ if } v_{\alpha} = -, \quad 0 < \langle x|\alpha\rangle < 1 \textrm{ if } v_{\alpha} = 0, \quad 1 < \langle x|\alpha\rangle \textrm{ if } v_{\alpha} = +.\]
The \emph{floors} of $R$ are the roots $\alpha \in \Phi^+$ such that the inequality $1 < \langle x|\alpha\rangle$ is irredundant in the above set of inequalities (equivalently, such that $H_{\alpha, 1}$ contains a facet of $\overline{R}$ and $R \subset H^+_{\alpha, 1}$). The set of floors of $R$ is denoted by $fl(R)$.

The philosophy of Athanasiadis \& Linusson's bijection is to encode a Shi region by two things: an element $w \in W$ and the floors of $R$. The element $w$ allows to situate $R$ with respect to the hyperplanes $H_{\alpha, 0}$, while the floors give the missing information about the (relevant) $H_{\alpha, 1}$. The fact that these floors form a non-nesting partition is not a "type $A_n$ miracle" and the objective of this paragraph is to state a result from Armstrong, Reiner and Rhoades \cite{armstrong2015parking} that generalize this labeling of Shi regions to the other crystallographic groups. To that end, we need to define the notions of non-nesting partition (following Postnikov \cite[Remark 2]{reiner1997non}), and of parking functions (using a slightly modified, low-tech version of the definition from \cite{armstrong2015parking}) in other Weyl types.

\begin{dftn}
	Recall the definition of root poset from \S\ref{sub:affine}. A \emph{non-nesting partition} (of type $W$) is an antichain of the root poset $\Phi^+$.
\end{dftn}

Importantly, this definition coincides with the "pictorial" notion of non-nesting partition in type $A_n$. The question of interpreting this definition in the "pictorial" sense used in Section \ref{sec:A} for the types $B_n, C_n, D_n$ is discussed in the next paragraph.

\begin{dftn}[Parking function]
	Let $W$ be a Weyl group and $\Phi^+$ the set of its positive roots. A \emph{(type $W$) parking function} is a pair $(w, P)$ where $P$ is a non-nesting partition of type $W$ and $w$ is an element of $W$ such that for all $\alpha \in P, w(\alpha) \in \Phi^+$.
\end{dftn}

\begin{thm}\cite[Proposition 10.3]{armstrong2015parking}\label{thm:ARR}
	Let $W$ be a Weyl group.
	\begin{itemize}
		\item Let $R$ be a region of the Shi arrangement of $W$ and $w$ the unique element of $W$ such that $R \subset wC$. Then $w\inv fl(R)$ is a non-nesting partition.\footnote{Note that in \cite{armstrong2015parking}, the authors actually prove the result for the \emph{ceilings} of $R$, that is the roots $\alpha$ such that the inequality $\langle v | \alpha\rangle < 1$ for $v \in R$ is irredundant. However, the proof is the same, replacing "ceilings inequalities" with "floor inequalities" wherever they appear.}
		\item The map that associates to a region $R$ the pair $(w, w\inv fl(R))$ is a bijection between Shi regions and parking functions.
	\end{itemize}
\end{thm}

Comparing this result with Theorem \ref{thm:AL}, the roots in $w\inv(fl(R))$ can be seen as the arcs of the arc diagram. The use of $w\inv$ comes from the fact that two arcs are nesting or not only depends on the positions of their extremities: the root $\alpha$ gives the values of the arc while $w\inv(\alpha)$ gives its positions.

\subsection{Types $B_n, C_n, D_n$}
\begin{dftn}
	The root systems described in Table \ref{fig:classical_root} are called the \emph{classical root systems}.
	\begin{table}[h!]
		\begin{tabular}{ c | c | c }
			Type & $\Phi^+$ & $\Delta$\\
			\hline
			$A_n$ & $e_i-e_{j}$ for $i, j \in \intint{1,n}, i < j$ & $e_i-e_{i+1}$ for $i \in \intint{1,n}$\\
			$B_n$ & $e_i\pm e_{j}$ for $i, j \in \intint{1,n}, i < j$, $e_i$ for $i \in \intint{1, n}$ & $e_n, e_i-e_{i+1}$ for $i \in \intint{1,n-1}$\\
			$C_n$ & $e_i\pm e_{j}$ for $i, j \in \intint{1,n}, i < j$, $2e_i$ for $i \in \intint{1, n}$ & $2e_n, e_i-e_{i+1}$ for $i \in \intint{1,n-1}$\\
			$D_n$ & $e_i\pm e_{j}$ for $i, j \in \intint{1,n}, i < j$ & $e_{n-1}+e_n, e_i-e_{i+1}$ for $i \in \intint{1,n-1}$\\
		\end{tabular}
		\caption{Classical root systems}
		\label{fig:classical_root}
	\end{table}
\end{dftn}

When examining the proof of Proposition \ref{prop:shi_min_A}, the essential element appears to be Lemma \ref{lem:arcs}. Thus we need to check mainly two things:
\begin{enumerate}
	\item In type $A_n$, we used the fact that a relation $\alpha+\beta=\gamma$ corresponds to a triple of integers in $\intint{1, n}$ on which we applied Lemma \ref{lem:arcs}. This implicitly used the realization of $A_n$ as a permutation group. As $B_n, C_n, D_n$ can also be realized as permutation groups of $\intint{-n, n}$, this extends to all classical groups, using the correspondence between roots and pairs of integers given in Table \ref{fig:rootsNarcs}.
	\begin{table}[h!]
		\centering
		\begin{tabular}{c|c|c|c|c}
			Root   & $e_i-e_j$ & $e_i+e_j$ & $2e_i$ & $e_i$\\
			\hline
			Extremities & $i$ to $j$ and $-j$ to $-i$ & $i$ to $-j$ and $j$ to $-i$ & $i$ to $-i$ & $i$ to $0$ and $0$ to $-i$
		\end{tabular}
		\caption{Roots and corresponding arcs in classical types. The table gives the \emph{positions} of the extremities of the arcs when constructing the arc diagram from a partition $P$, and the \emph{values} of the extremities when computing the minimal Shi vector.}
		\label{fig:rootsNarcs}
		\vspace{-4mm}
	\end{table}
	
	It is easy to check that all the relations between roots on which Theorem \ref{thm:shi_relations} applies correspond to suitable integer triples (for instance, the relation $e_i-e_j+e_i+e_j = 2e_i$ can correspond to $(i,j,-i)$ or $(-i,-j,i)$)
	\item Lemma \ref{lem:arcs} applies only for the "pictorial" notion on non-nesting arcs, meaning for two arcs $[a, b], [c,d]$ with $a<b, c<d$ and $a<c$ we don't have $a<c<d<b$. Thus, to use Theorem \ref{thm:ARR}, we need to get a "pictorial" representation of the "antichain" definition of non-nesting. We discuss this point below. In particular, this requires an extension of Lemma \ref{lem:arcs} in the case $D_n$.
\end{enumerate}

\paragraph{Type $B_n, C_n$:}
	The groups $B_n$ and $C_n$ are isomorphic to the permutation group:
	\[\{\pi \in \symm_{\intint{-n,n}} \sep \forall i \in \intint{1,n}, \pi(i) = -\pi(-i)\}.\]
	Note that although $B_n$ and $C_n$ are isomorphic, they require a slightly different treatment as their Shi arrangements are not the same.
	Given such a permutation $\pi$, we write it as the sequence of its values in the following order:
	\[\begin{matrix}
	\pi(1)&\pi(2)&\dots&\pi(n)&0&\pi(-n)&\dots&\pi(-2)&\pi(-1)
	\end{matrix}\]
	By convention, we say that the identity permutation written in this format is sorted.
	We define arc diagrams as before: given a Shi region $R$, we compute the corresponding parking function $(\pi, P)$, we write $\pi$ as above, and then, for every root in $P$, we draw arcs between \emph{positions} encoded by the root as precised above in Table \ref{fig:rootsNarcs}: the difference between $B_n$ and $C_n$ is reflected in the fact that $2e_i$ and $e_i$ do not correspond to the same arcs.
	It is known (see for instance \cite{athanasiadis1998noncrossing}) that, using these conventions, a non-nesting partition in the "antichain" sense corresponds to a non-nesting partition in the "pictorial" sense as we have used in Section \ref{sec:A}. 

	Given $i, j \in \intint{-n, n}$ and for $T \in \{B_n, C_n\}$, we define $\eta^{T}_{i,j}$ as the maximal number of non-nesting non-crossing arcs that can be chosen between the $i$ and $j$.

\paragraph{Type $D_n$:}
The group $D_n$ is isomorphic to the permutation group:
\[\{\pi \in B_n \sep \{i \in \intint{1, n}\sep \pi(i) < 0\}\textrm{ has even cardinality.}\}.\]
Given such a permutation $\pi$, we will write it as the sequence of its values in the following order:
\begin{figure}[h!]
	\vspace{-3mm}
	\centering
	\begin{tikzpicture}
	\draw (-6, 0)node{$\pi(1)$};
	\draw (-4.5, 0)node{$\pi(2)$};
	\draw (-3.5, 0)node{$\cdots$};
	\draw (-2, 0)node{$\pi(n-1)$};
	\draw (0,0.5)node{$\pi(n)$};
	\draw (0,-0.5)node{$\pi(-{n})$};
	\draw (2, 0)node{$\pi(-{n+1})$};
	\draw (3.5, 0)node{$\cdots$};
	\draw (4.5, 0)node{$\pi(-{2})$};
	\draw (6, 0)node{$\pi(-{1})$};
	\end{tikzpicture}
\end{figure}\\
Again, by convention, we say that the identity permutation written in this format is sorted, meaning $n$ and $-{n}$ are uncomparable, both greater than $n-1$ and both lesser than $-{n+1}$. We do this to "solve" an issue noted in \cite{athanasiadis1998noncrossing}: if we simply ordered the elements of $\intint{-n, n}$ as we did in case $B_n, C_n$, the antichain $\{e_{n-1}-e_n, e_{n-1}+e_n\}$ in type $D_n$ would correspond to two nesting arcs. By making $n$ and $-n$ incomparable we preserve the correspondence between the "antichain" definition of non-nesting and the "pictorial" definition. 
This has the advantage to (almost) allow us the use of Lemma \ref{lem:arcs} when we define arc diagrams as before: for $(\pi, P)$, write $\pi$ in the above format, then draw arcs as indicated by Table \ref{fig:rootsNarcs}. The drawback is that the definition of the arc-counting vector $\eta$ is slightly more complicated. See Figure \ref{fig:Darcs} for a sketch of a proof the following Lemma.

\begin{lemme}\label{lem:Darcs}
	Let $(\Id_{\intint{-n, n}}, P)$ be a non-nesting arc diagram of type $D_n$. Let $a,b \in \intint{-n, n}\setminus\{0\}$. We define $\eta_{a,b}^+$ (resp. $\eta_{a,b}^-$) as the maximal number of non-nesting, non-crossing arcs between $a$ and $b$, excluding arcs connected to $-n$ (resp. to $n$).
	Define $\eta^D_{a,b} = \max(\eta_{a,b}^+, \eta_{a,b}^-)$. Then we have $\eta^D_{a,c} = \eta^D_{a,b} + \eta^D_{b,c} + \varepsilon_{\eta}$, with $\varepsilon_{\eta} \in \{0,1\}$.
\end{lemme}

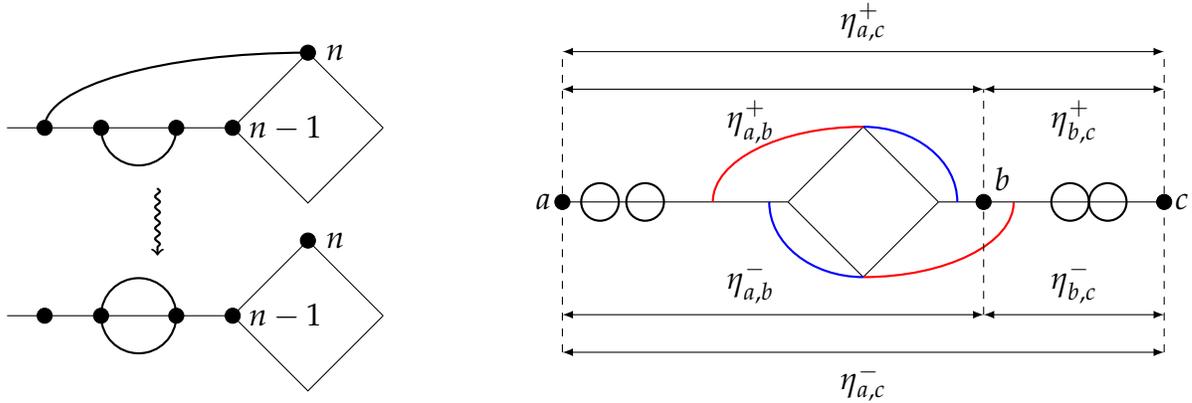
\begin{figure}[h!]
	\centering
	\begin{subfigure}[b]{0.4\textwidth}
	\centering
	\begin{tikzpicture}
	\draw (-1,0)--(2,0);
	\draw (2,0)--++(1,1)--++(1,-1);
	\draw (2,0)--++(1,-1)--++(1,1);
	\draw (3,1)[thick] arc
	[
	start angle=90,
	end angle=180,
	x radius=3.5cm,
	y radius =1cm
	];
	\draw (0.25,0)[thick] arc
	[
	start angle=180,
	end angle=360,
	radius = 0.5cm
	];
	\draw[circle,fill](-0.5,0)circle[radius=1mm];
	\draw[circle,fill](0.25,0)circle[radius=1mm];
	\draw[circle,fill](1.25,0)circle[radius=1mm];
	\draw[circle,fill](2,0)circle[radius=1mm]node[right]{$n-1$};
	\draw[circle,fill](3,1)circle[radius=1mm]node[right]{$n$};
	
	\draw [->,
	line join=round, thick,
	decorate, decoration={
		snake,
		segment length=4,
		amplitude=.9, post=lineto,
		post length=2pt
	}]  (1,-.8) -- (1,-1.7);
	
	\def\y{-2.5};
	\draw (-1,0+\y)--++(3,0);
	\draw (2,0+\y)--++(1,1)--++(1,-1);
	\draw (2,0+\y)--++(1,-1)--++(1,1);
	\draw (1.25,0+\y)[thick] arc
	[
	start angle=0,
	end angle=360,
	radius = 0.5cm
	];
	\draw[circle,fill](-0.5,0+\y)circle[radius=1mm];
	\draw[circle,fill](0.25,0+\y)circle[radius=1mm];
	\draw[circle,fill](1.25,0+\y)circle[radius=1mm];
	\draw[circle,fill](2,0+\y)circle[radius=1mm]node[right]{$n-1$};
	\draw[circle,fill](3,1+\y)circle[radius=1mm]node[right]{$n$};
\end{tikzpicture}
	\vspace{3mm}
	\caption{Except in $\pm n$, a symmetric choice of arcs to compute $\eta_{a, b}$ can be made, thus...}
	\label{sfig:symm_choice}
	\end{subfigure}
	\qquad
	\begin{subfigure}[b]{0.5\textwidth}
	\centering
	\begin{tikzpicture}
	\draw (-1,0)--(2,0);
	\draw (2,0)--++(1,1)--++(1,-1);
	\draw (2,0)--++(1,-1)--++(1,1);
	\draw (4,0)--++(3,0);
	
	\draw (3,1)[red, thick] arc
	[
	start angle=90,
	end angle=180,
	x radius=2cm,
	y radius =1cm
	];
	\draw (3,-1)[red, thick] arc
	[
	start angle=270,
	end angle=360,
	x radius=2cm,
	y radius =1cm
	];
	\draw (3,-1)[blue, thick] arc
	[
	start angle=270,
	end angle=180,
	x radius=1.25cm,
	y radius =1cm
	];
	\draw (3,1)[blue, thick] arc
	[
	start angle=90,
	end angle=0,
	x radius=1.25cm,
	y radius =1cm
	];
	\draw (0.35,0)[black, thick] arc
	[start angle=0, end angle=360, radius = 0.25cm];
	\draw (-0.25,0)[black, thick] arc
	[start angle=0, end angle=360, radius = 0.25cm];
	\draw (6,0)[black, thick] arc
	[start angle=0, end angle=360, radius = 0.25cm];
	\draw (6.5,0)[black, thick] arc
	[start angle=0, end angle=360, radius = 0.25cm];
	
	\draw[dashed] (-1,-2)--(-1,2);
	\draw[dashed] (4.6,-1.5)--(4.6,1.5);
	\draw[dashed] (7,-2)--(7,2);
	
	\draw[latex-latex] (-1, 2)--(7,2);
	\draw(3, 2)node[above]{$\eta^+_{a,c}$};
	\draw[latex-latex] (-1, -2)--(7,-2);
	\draw(3, -2)node[below]{$\eta^-_{a,c}$};
	
	\draw[latex-latex] (-1, 1.5)--(4.6,1.5);
	\draw(1.5, 1.5)node[below]{$\eta^+_{a,b}$};
	\draw[latex-latex] (4.6, 1.5)--(7,1.5);
	\draw(5.8, 1.5)node[below]{$\eta^+_{b,c}$};
	
	\draw[latex-latex] (-1, -1.5)--(4.6,-1.5);
	\draw(1.5, -1.5)node[above]{$\eta^-_{a,b}$};
	\draw[latex-latex] (4.6, -1.5)--(7,-1.5);
	\draw(5.8, -1.5)node[above]{$\eta^-_{b,c}$};
	
	\draw[circle,fill](-1,0)circle[radius=1mm];
	\draw(-1,0)node[left]{$a$};
	\draw[circle,fill](4.6,0)circle[radius=1mm];
	\draw(4.6,0)node[above right]{$b$};
	\draw[circle,fill](7,0)circle[radius=1mm];
	\draw(7,0)node[right]{$c$};
\end{tikzpicture}
	\vspace{-5mm}
	\caption{... allowing to reason on this kind of picture. Different cases corresponding to the positions of $a,b,c$ must be checked.}
	\label{sfig:beads}
	\end{subfigure}
\caption{Idea of proof for Lemma \ref{lem:Darcs}.}
\label{fig:Darcs}
\end{figure}



\paragraph{Describing the minimal elements.}
Given the preceding points, the proof of the following Proposition is, with small caveats, essentially the same as in type $A_n$.
	
\begin{prop}\label{prop:shi_min_BCD}
	Let $R$ be a Shi region of type $T$, $T \in \{B_n, C_n, D_n\}$ of sign type $v= \sign(R)$.
	Let $(\pi, P)$ be the arc diagram of type $T$ associated to $v$.
	We define $m$ as:
	\[m_{\alpha} = \left\{\begin{aligned}
	\eta^T_{i,j} &\textrm{ if } v_{\alpha} = 0, +\\
	-(\eta^T_{i,j}+1) &\textrm{ if } v_{\alpha} = - \\
	\end{aligned}\right.\]
	where $i, j$ are given by Table \ref{fig:classical_root} depending on $\alpha$. Then $m$ is $R$'s minimal element.
\end{prop}


%

As in type $A_n$, this construction together with the bijection from Theorem \ref{thm:ARR} give an explicit bijection between parking functions of type $T$ and the minimal elements of the Shi regions of type $T$.

\section{Final remarks}

Two main questions arise when considering this construction. Firstly, how to extend it to exceptional types? An analysis of this arc counting method prompts a formulation only in terms of roots: given an antichain $P$, $\eta_{\alpha}$ can be defined as the maximal possible number of occurences of elements of $P$ when writing $\alpha$ as a sum of positive roots. In the classical types this corresponds to the definition used in this paper, and it still makes sense in all crystallographic types. However, to prove that this definition verifies Lemma \ref{lem:arcs}, we have strongly used the permutation groups realizations in the classical types but no equivalent realizations exists in the exceptional cases. As of now, it appears that further examination of the root poset is needed in this context.

Secondly, for $m \in \NN_{>0}$, a $m$-Shi arrangement can be defined. Can this construction be extended to the $m$-Shi arrangements? In \cite{athanasiadis1999simple}, Athanasiadis \& Linusson extend their bijection the $m$-Shi regions. Due to the similar use of non-nesting arc diagrams, is seems that our description of the minimal elements can be extended to $m$-Shi arrangements, at least in type $A_n$.

\section*{Acknowledgments}
This work was impulsed in a working group with Antoine Abram, Nathan Chapelier-Laget and Elias Thouant, whom the author thanks. The author also wants to thank Hugo Mlodecki and Daniel Tamayo-Jiménez for many interesting discussions on this matter as well as renew his thanks to Nathan Chapelier-Laget for providing many references and writing advice.

\printbibliography

\end{document}